\documentclass[a4paper,reqno,11pt]{amsart}

\usepackage[T1]{fontenc}
\usepackage[linktocpage]{hyperref} 	\hypersetup{colorlinks=true,linkcolor=blue,citecolor=blue,filecolor=magenta,urlcolor=blue}
\usepackage{geometry}       
\usepackage[english]{babel}  
\usepackage{amsmath,amsfonts,amssymb,amsthm}
\usepackage{tikz,tikz-cd}
\usepackage{enumitem}

\newtheorem{theorem}{Theorem}[section]
\newtheorem{proposition}[theorem]{Proposition}
\newtheorem{lemma}[theorem]{Lemma}
\newtheorem{corollary}[theorem]{Corollary}

\theoremstyle{definition}
\newtheorem{definition}[theorem]{Definition}                  
\newtheorem{example}[theorem]{Example}

\theoremstyle{remark}
\newtheorem{remark}[theorem]{Remark}

\newcommand{\CC}{\mathbb{C}}
\newcommand{\NN}{\mathbb{N}}
\newcommand{\ZZ}{\mathbb{Z}}
\newcommand{\PP}{\mathbb{P}}
\newcommand{\wPP}{{\PP^\vee}}
\newcommand{\A}{\mathcal{A}}
\newcommand{\B}{\mathcal{B}}
\renewcommand{\P}{\mathcal{P}}
\newcommand{\D}{\mathcal{D}}
\newcommand{\E}{\mathcal{E}}
\newcommand{\R}{\mathcal{R}}
\newcommand{\C}{\mathcal{C}}
\newcommand{\M}{\mathcal{M}}
\renewcommand{\L}{\mathcal{L}}

\DeclareMathOperator{\Sing}{Sing}

\DeclareMathOperator{\Irr}{Irr}
\DeclareMathOperator{\PGL}{PGL}
\DeclareMathOperator{\Arr}{Arr}
\DeclareMathOperator{\naive}{d^{naive}}
\DeclareMathOperator{\St}{St}
\DeclareMathOperator{\Net}{Net}
\DeclareMathOperator{\bv}{\mathbf{v}}
\DeclareMathOperator{\compo}{CC}

\newcommand{\set}[1]{\left\{ #1 \right\}}
\newcommand{\gen}[1]{\left\langle #1 \right\rangle}
\newcommand{\X}{\mathfrak{X}}
\newcommand{\ol}[1]{\overline{#1}}
\newcommand{\pprec}{\prec\mkern-5mu\prec}

\makeatletter
\let\origsection\section
\renewcommand\section{\@ifstar{\starsection}{\nostarsection}}

\newcommand\nostarsection[1]
{\sectionprelude\origsection{#1}\sectionpostlude}

\newcommand\starsection[1]
{\sectionprelude\origsection*{#1}\sectionpostlude}

\newcommand\sectionprelude{%
  \vspace{2em}
}

\newcommand\sectionpostlude{%
  \vspace{1em}
}
\makeatother

\makeatletter
\let\origsubsection\subsection
\renewcommand\subsection{\@ifstar{\starsubsection}{\nostarsubsection}}

\newcommand\nostarsubsection[1]
{\subsectionprelude\origsubsection{#1}\subsectionpostlude}

\newcommand\starsubsection[1]
{\subsectionprelude\origsubsection*{#1}\subsectionpostlude}

\newcommand\subsectionprelude{%
  \vspace{0.5em}
}

\newcommand\subsectionpostlude{%
  \vspace{0em}
}
\makeatother

\makeatletter
\let\origsubsubsection\subsubsection
\renewcommand\subsubsection{\@ifstar{\starsubsubsection}{\nostarsubsubsection}}

\newcommand\nostarsubsubsection[1]
{\subsubsectionprelude\origsubsubsection{#1}\subsubsectionpostlude}

\newcommand\starsubsubsection[1]
{\subsubsectionprelude\origsubsubsection*{#1}\subsubsectionpostlude}

\newcommand\subsubsectionprelude{%
  \vspace{0.5em}
}

\newcommand\subsubsectionpostlude{%
  \vspace{0em}
}
\makeatother

\begin{document}

\title{Connectedness and combinatorial interplay in the moduli space of line arrangements}

\author[B. Guerville-Ball\'e]{Beno\^it Guerville-Ball\'e}
			 \address[B. Guerville-Ball\'e]{}
			 \urladdr{\url{http://www.benoit-guervilleballe.com/}}
			 \email{\href{benoit.guerville-balle@math.cnrs.fr}{benoit.guerville-balle@math.cnrs.fr}}
			 
\author[J.~Viu-Sos]{Juan Viu-Sos}
\address[J.~Viu-Sos]{
	Dpto. de Matemáticas e Informática (DMIAICN)\newline\indent
    ETSI Caminos, Canales y Puertos\newline\indent
    Universidad Politécnica de Madrid\newline\indent
    C\textbackslash Prof. Aranguren 3, 28040 Madrid, Spain
}
\urladdr{\url{https://jviusos.github.io/}}
\email{\href{mailto:juan.viu.sos@upm.es}{juan.viu.sos@upm.es}}

\thanks{The first author was supported by RIMS and JSPS KAKENHI Grant JP23H00081 (PI: Masahiko Yoshinaga). The second author was supported by the Spanish grants PID2020-114750GB-C32 and MTM2016-76868-C2-2-P}				


\begin{abstract}
	This paper aims to undertake an exploration of the behavior of the moduli space of line arrangements while establishing its combinatorial interplay with the incidence structure of the arrangement. In the first part, we investigate combinatorial classes of arrangements whose moduli space is connected. We unify the classes of simple and inductively connected arrangements appearing in the literature. Then, we introduce the notion of arrangements with a rigid pencil form. It ensures the connectedness of the moduli space and is less restrictive that the class of $C_3$ arrangements of simple type. In the last part, we obtain a combinatorial upper bound on the number of connected components of the moduli space. Then, we exhibit examples with an arbitrarily large number of connected components for which this upper bound is sharp.
\end{abstract}

\maketitle

\setcounter{tocdepth}{1}
\tableofcontents

\section{Introduction}

Complex line arrangements are fundamental objects in algebraic geometry, algebraic topology, and also combinatorics. The moduli space of a line arrangement $\A$ is a geometrical object which describes isomorphism classes, under the natural action of $\PGL_3(\CC)$, of arrangements that are lattice-isomorphic to~$\A$. The topological study of the moduli space is extremely important in this context, e.g.~arrangements in the same connected component of the moduli space have diffeomorphic complements~\cite{Ran:lattice_isotopy}. In this paper, we focus on a combinatorial study of the moduli space of line arrangements and we explore how such a study can provide fundamental information about its geometry and topology.

\medskip

As proved by Mn\"ev in~\cite{Mnev}, the moduli space of a line arrangement can behave as badly as one can imagine. Vakil qualifies such behavior as \emph{Murphy's law}~\cite{Vak}. More precisely, Mn\"ev Universality Theorem states that every singularity of finite type over $\ZZ$ appears in at least one moduli space. Recently, it was shown in~\cite{CL:singular_moduli} that the realization space of line arrangements is smooth for $|\A|\leq 11$ and then nodal singularities appear for $|\A|=12$. In addition, there are several basic geometric or topological aspects of the moduli space that can not be fully predicted by a combinatorial study, as it is shown by the classical Pappus’ hexagon theorem. Thus, a natural question follows: which property of the moduli space can be combinatorially deduced?\\

We recall, in Section~\ref{sec:arr_combi}, the definitions of the realization space and moduli space of a line arrangement. Also,  we develop a systematic study of the inductive construction of the moduli space of ordered arrangements by introducing the notion of the \emph{type} (Definition~\ref{def:type}). Roughly speaking, it encodes the combinatorial complexity of the line-by-line construction of an ordered arrangement by counting the multiple points contained in the new line at each step.

Section~\ref{sec:families} is dedicated to the investigation of several known combinatorial classes of arrangements with connected moduli spaces: e.g.~the class of \emph{nice} arrangements~\cite{JiaYau} and their generalization to \emph{simple} arrangements~\cite{WangYau}, or the one of \emph{inductively connected} arrangements and the \emph{$C_3$ arrangements of simple type} both introduced in~\cite{NazirYoshinaga12}. %
We prove in Theorem~\ref{thm:simple_arrangement} that any generalized simple arrangement is inductively connected, answering a question from Nazir and Yoshinaga (see the introduction of~\cite{NazirYoshinaga12}). In the last part of this section, we first introduce the combinatorial class of \emph{inductively rigid} arrangements whose moduli space is a unique point. Using this as a substructure, we define the class of arrangements with a \emph{rigid pencil form} (Definition~\ref{def:rigid_pencil}), which contains all the $C_3$ arrangements of simple type. Likewise, any arrangement in this class has a connected moduli space (Theorem~\ref{thm:C3_generalization}).

In the latter part of this paper (Section~\ref{sec:irred_components}), we introduce an upper bound on the number of connected components of the moduli space of an arrangement. We first define the notion of \emph{$m$-perturbation} which transforms any combinatorics into an inductively connected one (Definition~\ref{def:m_perturbation}). By the work of Nazir and Yoshinaga~\cite{NazirYoshinaga12}, this allows to embed the moduli space of our original arrangement in an affine space of a relatively small dimension (Theorem~\ref{thm:moduli_embedding}). From this embedding, we derive an upper bound of the number of connected components of the moduli space (Theorem~\ref{thm:upper_bound}) by inductively bounding the degree of the $m$ polynomial conditions arising from the $m$-perturbation. Finally, by providing explicit examples, we prove that the upper bound obtained is sharp even for non-trivial cases for which the number of connected components is arbitrarily large (Theorem~\ref{thm:sharpness}).


\section{Moduli spaces of line arrangements}\label{sec:arr_combi}

Throughout all the paper, a line arrangement $\A$ refers to a finite set of lines $\{\ell_1,\dots,\ell_n\}$ in the complex projective plane $\CC\PP^2$. In this section, we will recall some basic definitions about moduli space of line arrangements. We will also fix some notations and define the type of an ordered arrangement and the naive dimension of its moduli space.

\subsection{Line combinatorics and line arrangements}\mbox{}

\begin{definition}\label{def:combinatorics}
	An \emph{abstract line combinatorics} is a pair $\C=(\L,\P)$ where:
	\begin{itemize}
	  \item $\L$ is the data of an ordered set of $n$ \emph{lines}, codified by the list $\L=\set{1,\ldots,n}$,
	  \item $\P$ is a finite subset of the power set of $\L$ codifying the \emph{singular points}, satisfying:
	\begin{enumerate}
	  \item for all $P\in\P$, one has $|P|\geq 2$,
	  \item for all $i,j\in\L$, there exists a unique $P\in\P$ such that $i,j\in P$.
	\end{enumerate}
	\end{itemize}
	Two abstract line combinatorics $\C=(\L,\P)$ and $\C'=(\L',\P')$ are called \emph{equivalent}, denoted $\C\simeq\C'$, if there exists a (non-necessarily order-preserving) bijection
	$\Phi:\L\to\L'$ such that $\Phi(\P)=\P'$.
\end{definition}

\begin{remark}
	Whenever there is no confusion, we will summarize an abstract line combinatorics $\C$ by describing only the set $\P$.
\end{remark}

Let $\Arr_n$ be the set of all the line arrangements of $n$ lines. Given $\A\in\Arr_n$, its abstract line combinatorics $\C(\A)$ (or shortly its \emph{combinatorics}) is given by $(\A,\P)$, with
\begin{equation*}
	\P=\left\{
		P\subset \A \,\left\vert\, \bigcap_{L\in P}L \neq \emptyset\;\text{ and }\; \forall L'\in \A\setminus P:\; L'\cap \bigcap_{L\in P} L  = \emptyset \right.
	\right\}.
\end{equation*}
A \emph{multiple point} of $\C(\A)$ is an element $P\in\P$ such that $|P|\geq 3$. Singular points which are not multiple points are called \emph{double points}. An abstract line combinatorics $\C$ is just called \emph{line combinatorics} if there exists an arrangement $\A$ such that $\C\simeq\C(\A)$. By Pappus' hexagon theorem, we know there exist abstract line combinatorics which are not line combinatorics.

\begin{example}\label{ex:Ceva}
	Let $\A$ be the Ceva arrangement defined by the following equations:
	\[   
	  \ell_1: x=0,\quad \ell_2: y=0,\quad \ell_3: z=0,\quad \ell_4: x-y=0,\quad \ell_5: x+z = 0\quad \&\quad \ell_6: y+z=0.
	\]
	Its combinatorics is then:
	\[   
	  \C(\A)=\set{\set{\ell_1,\ell_2,\ell_4},\, \set{\ell_1,\ell_3,\ell_5},\, \set{\ell_1,\ell_6},\, \set{\ell_2,\ell_3,\ell_6},\, \set{\ell_2,\ell_5},\, \set{\ell_3,\ell_4}, \set{\ell_4,\ell_5,\ell_6} }.
	\]
\end{example}

In the following, we will consider two particular families of line arrangements. Let $n$ be a positive natural number. We define:
\begin{itemize}
  \item the family of arrangements formed by a pencil of $n$ lines, denoted by $\X(n)\subset\Arr_n$,
  \item the family of arrangements formed by a pencil of $n-1$ lines and a unique generic line, denoted by $\ol{\X}(n)\subset\Arr_n$.
\end{itemize}

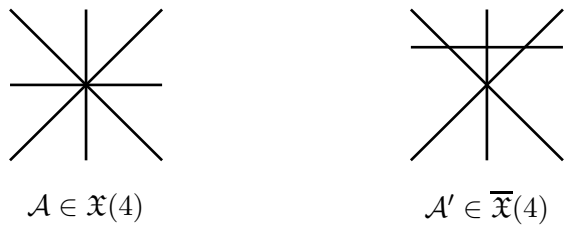
\begin{figure}[h]
	\begin{tikzpicture}[scale=0.5]
		\tikzstyle{droites}=[line width=1pt]
		\begin{scope}
			\draw[droites] (-2,2) -- (2,-2);
			\draw[droites] (-2,0) -- (2,0);
			\draw[droites] (-2,-2) -- (2,2);
			\draw[droites] (0,-2) -- (0,2);
			\node at (0,-3.25) {$\A\in\X(4)$};
		\end{scope}
		\begin{scope}[xshift=300]
			\draw[droites] (-2,2) -- (2,-2);
			\draw[droites] (-2,1) -- (2,1);
			\draw[droites] (-2,-2) -- (2,2);
			\draw[droites] (0,-2) -- (0,2);
			\node at (0,-3.25) {$\A'\in\ol{\X}(4)$};
		\end{scope}
	\end{tikzpicture}
	\caption{Arrangements in the classes $\X(4)$ and $\ol{\X}(4)$.\label{fig:Xn}}
\end{figure}

\subsection{Realization and moduli spaces}\label{sec:moduli_space}\mbox{}

\begin{definition}\label{def:realization_space}
	Let $\C=(\L,\P)$ be an abstract line combinatorics, and let $n$ be the cardinality of $\L$.
	The \emph{realization space} of $\C$ is the set:
	\begin{equation*}
		\R(\C) =  \left\{ \B \in \Arr_n \mid \C(\B)\simeq \C \right\}.
	\end{equation*}
	The \emph{moduli space} $\M(\C)$ of the abstract line combinatorics $\C$ is the quotient of the realization space $\R(\C)$ by the action of $\PGL_3(\CC)$. For brevity sake, we set $\R(\A):=\R(\C(\A))$ and $\M(\A):=\M(\C(\A))$ for any $\A\in\Arr_n$ with combinatorics $\C(\A)$.
\end{definition}

A line $\ell: a x + b y + c z = 0$ in $\CC\PP^2$ %
can be consider as a point $(a:b:c)$ in the dual $\CC\wPP^2$, then an element of $\Arr_n$ is naturally identified with a point in $\big(\CC\wPP^2\big)^n$. 
Three different lines $\ell_i = (a_i:b_i:c_i)$, $\ell_j= (a_j:b_j:c_j)$ and $\ell_k = (a_k:b_k:c_k)$ are \emph{concurrent} if and only if
\begin{equation*}
	\Delta_{i,j,k}:=\det(\ell_i,\ell_j,\ell_k)=
	\begin{vmatrix}
		a_i & a_j & a_k \\
		b_i & b_j & b_k \\
		c_i & c_j & c_k 
	\end{vmatrix} = 0.
\end{equation*}
By definition, the {realization space} of a line combinatorics $\C(\A)$ with $\A\in\Arr_n$ can be constructed as a subset of $\big(\CC\wPP^2\big)^n$ as follows:
\begin{equation*}
	\R(\A)=\left\{
	\begin{array}{c|ll}
		& \ell_i \neq \ell_j, & \forall i\neq j \\
		(\ell_1,\dots,\ell_n) \in \big(\CC\wPP^2\big)^n & \Delta_{i,j,k}=0, & \text{if $\{i,j,k\}\subset P$, for a $P\in\P$} \\
		& \Delta_{i,j,k} \neq 0, & \text{otherwise} \\
	\end{array}
	\right\}.
\end{equation*}

If $\A$ is not in $\X(n)$ or $\ol{\X}(n)$, then it contains four lines in general position. Up to a relabelling, one can assume that these lines are $\ell_1,\ldots,\ell_4$. Any isomorphism class in $\M(\A)$ admits a unique representative\footnote{Throughout this paper any isomorphism class of $\M(\A)$ is identified with this unique representative.} in $\R(\A)$ such that $\ell_1,\ldots,\ell_4$ is the arrangement $S_0$ formed by the lines $x=0$, $x-z=0$, $y=0$ and $y-z=0$. So the moduli space $\M(\A)$ can be considered in $\big(\CC\wPP^2\big)^{n}$ as the trivial fibration over the point $S_0\in\big(\CC\wPP^2\big)^{4}$ as follows:
\begin{equation}\label{eqn:moduli_space}
	\M(\A)= \{S_0\} \times
	\left\{
	\begin{array}{c|ll}
		& \ell_i \neq \ell_j, & \forall i\neq j \in\{1,\dots,n\}\\
		(\ell_5,\dots,\ell_n) \in \big(\CC\wPP^2\big)^{n-4} & \Delta_{i,j,k}=0, & \text{if $\{i,j,k\}\subset P$, for a $P\in\P$} \\
		& \Delta_{i,j,k} \neq 0, & \text{otherwise} \\
	\end{array}
	\right\}.
\end{equation}

It follows that both $\R(\A)$ and $\M(\A)$ are (quasi-projective, non-necessarily irreducible) algebraic varieties and so the notions of \emph{dimension} and \emph{irreducible components} are well defined.

\begin{remark}\label{rmk:moduli_pencils}
  Let $\A$ be in $\X(n)$ (resp.~in $\ol{\X}(n)$), the moduli space $\M(\A)$ is irreducible of dimension $\max(0,n-3)$ (resp. $\max(0,n-4)$). If $\A$ is neither in $\X(n)$ nor in $\ol{\X}(n)$, one has that $\dim_\CC\M(\A) = \dim_\CC \R(\A)-8$.
\end{remark}

\subsection{Type of an arrangement and naive dimension of the moduli space}\label{sec:type}\mbox{}

An \emph{ordered arrangement} is a pair $(\A,\omega)$, where $\A\in\Arr_n$ and $\omega$ is a bijective map $$\omega: \A\longrightarrow\set{1,\ldots,n},$$ called the \emph{order} on $\A$. For the sack of brevity, we will sometime express an order arrangement as an ordered list $\A=\set{\ell_1,\ldots, \ell_n}$ where the indices of the lines are chosen in a compatible way with the order, i.e. such that $\omega(\ell_i)=i$, for any $\ell_i\in\A$.

Let $(\A,\omega)$ be an ordered arrangement and $\B$ be a subarrangement of $\A$. There exists a unique order $\omega_\B$ on $\B$ such that for any $\ell\neq\ell'\in\B$, one has that
\[
	\omega_\B(\ell) < \omega_\B(\ell') \iff \omega(\ell) < \omega(\ell').
\]
This order will be called the \emph{order induced} by $\omega$ on $\B$.

\medskip

Let $\A_i=\omega^{-1}(\set{1,\dots,i})$, for $i\in\set{1,\dots,n}$. When $\omega$ is the order given by the indices, then $\A_i=\set{\ell_1,\dots,\ell_i}$. Consider the following ascending chain of (ordered) arrangements:
\begin{equation}\label{eqn:ordered_chain}
  \A_1=\set{\ell_1}\subsetneq \A_2 \subsetneq \cdots \subsetneq \A_i \subsetneq \cdots \subsetneq \A_n = \A,
\end{equation}
We denote by $\tau_i(\A,\omega)$ (alternatively, $\tau(\A,\omega,\ell_i)$ or $\tau_i$ depending on the context) the cardinality of the intersection $|\omega^{-1}(i) \cap \Sing(\A_{i-1})|$. Roughly speaking, $\tau_i$ counts the number of multiple points of $\A_i$ contained in the $i$th line. Note that we also have:
\begin{equation}\label{eqn:taui_asdiff}
  \tau_i = \sum_{P\in\C(\A_i)} (|P|-2)\, - \sum_{Q\in\C(\A_{i-1})} (|Q|-2).
\end{equation}

\begin{definition}\label{def:type}
	The \emph{type} of an ordered arrangement $(\A,\omega)$ is the $n$-tuple defined as
	\[   
	  \tau(\A,\omega) = (\tau_1 , \dots, \tau_n) \in \NN^n.
	\]
\end{definition}

It is clear that the type depends on the choice of an order. In fact, both the ascending chain \eqref{eqn:ordered_chain} and the type are of combinatorial nature. So they can be defined directly over an abstract line combinatorics $\C$. By an extension of notation, we will also denote $(\C)_1\subsetneq\cdots\subsetneq(\C)_n=\C$ and $\tau(\C,\omega)$.

\begin{example}
	Consider the arrangement $\A$ pictured in Figure~\ref{fig:type} endowed with the order $\omega$ induced by the indices. Its type is $\tau(\A,\omega)=(0,0,0,0,2,2,1,3)$.
\end{example}

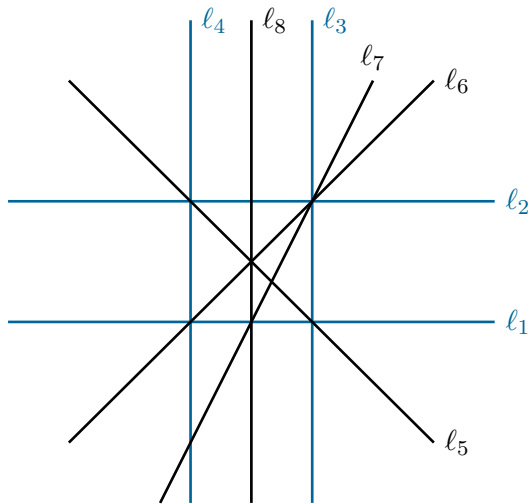
\begin{figure}[h]
	\begin{tikzpicture}[scale=0.8]
		\def\Pout{4}
		\tikzstyle{droites}=[line width=1pt]
		\draw[droites, blue!60!green] (-\Pout,-1) -- (\Pout,-1) node[right] {$\ell_1$};
		\draw[droites, blue!60!green] (-\Pout,1) -- (\Pout,1) node[right] {$\ell_2$};
		\draw[droites, blue!60!green] (1,-\Pout) -- (1,\Pout) node[right] {$\ell_3$};
		\draw[droites, blue!60!green] (-1,-\Pout) -- (-1,\Pout) node[right] {$\ell_4$};
		\draw[droites] (1-\Pout,\Pout-1) -- (\Pout-1,1-\Pout) node[right] {$\ell_5$};
		\draw[droites] (1-\Pout,1-\Pout) -- (\Pout-1,\Pout-1) node[right] {$\ell_6$};
		\draw[droites] (2.5-\Pout,-\Pout) -- (\Pout-2,\Pout-1) node[above] {$\ell_7$};
		\draw[droites] (0,-\Pout) -- (0,\Pout) node[right] {$\ell_{8}$};
	\end{tikzpicture}
	\caption{An arrangement with type $\tau(\A,\omega)=(0,0,0,0,2,2,1,3)$.\label{fig:type}}
\end{figure}

For two ordered arrangement $(\A,\omega)$ and $(\A',\omega')$ of $n$ and $n'$ lines respectively, and such that $\A\cap\A'=\emptyset$ (i.e. without common line, but not necessarily with generic intersection), one can define the operation $(\A\sqcup\A',\omega\oplus\omega')$ producing an ordered arrangements of $n+n'$ lines with order:
\[
  \begin{array}{rccl}
    \omega\oplus\omega' : & \A\sqcup\A'
		& \longrightarrow & \set{1,\ldots,n+n'}\\[0.3em]
		& \ell & \longmapsto &
			\begin{cases}
				\omega(\ell) & \text{if $\ell\in\A$},\\
				n+\omega'(\ell) & \text{if $\ell\in\A'$}.
			\end{cases}
  \end{array}
\]

\begin{remark}\label{rmk:tau_partition_independence}
	If $\A\in\Arr_n$ and $\A'\in\Arr_{n'}$ have no common lines, for any orders $\omega$ and $\omega'$ on $\A$ and $\A'$ respectively, one has that:
	\[
		\tau_i(\A\sqcup\A',\omega\oplus\omega') = \tau_i(\A,\omega)
		\quad\text{ and }\quad
		\tau_{n+j}(\A\sqcup\A',\omega\oplus\omega') \geq \tau_j(\A',\omega'),
	\]
	for any $i\in\{1,\dots,n\}$ and $j\in\{1,\dots,n'\}$. Furthermore, for any other order $\ol{\omega}$ on $\A$, and for any $j\in\{1,\dots,n'\}$, we have that:
	\[
		\tau_{n+j}(\A\sqcup\A',\omega\oplus\omega') = \tau_{n+j}(\A\sqcup\A',\ol{\omega}\oplus\omega').
	\]
\end{remark}

\begin{proposition}\label{propo:sumtau_dnaive}
	Let $(\A,\omega)$ be an ordered arrangement. The following equality holds:
	\begin{equation*}
		\sum_{i=1}^n (2-\tau_i) = 2|\A| - \sum_{P\in\C(\A)} (|P|-2).
	\end{equation*}
\end{proposition}

\begin{proof}
	First, if $|\A|=1$, then both sides of the equality are obviously $0$.

	Assume that the equality holds for any ordered arrangement of $n$ lines. Let $\A$ be an ordered arrangement of $n+1$ lines and take $\A'=\A\setminus\{\ell_{n+1}\}$. 
	By the induction hypothesis and equation~\eqref{eqn:taui_asdiff}, one has:
	\begin{equation*}
		\sum_{i=1}^{n+1} (2-\tau_i)
		=(2-\tau_{n+1}) + 2|\A'| - \sum_{Q\in\C(\A')} (|Q|-2)
		= 2|\A| - \sum_{P\in\C(\A)} (|P|-2).
		\qedhere
	\end{equation*}
\end{proof}

The right-hand side of the equality in the preceding proposition can be seen as a preliminary attempt to establish a combinatorial formula for the dimension of the realization space $\R(\A)$, despite its naive nature. The part $2|\A|$ gives the number of required variables in the construction, while the sum over $P\in\C(\A)$ corresponds to the restriction imposed by the singular points. Since the action of $\PGL_3(\CC)$ can fix four points in generic position, this induces a reduction of $8$ to pass from the dimension of $\R(\A)$ to the dimension of $\M(\A)$.

\begin{definition}\label{def:naive}
	Let $\A$ be a line arrangement. The \emph{naive dimension} of the moduli space is:
	\[
		\naive \M(\A) = 2|\A| - 8 - \sum_{P\in\C(\A)} (|P|-2).
	\]
\end{definition}

\begin{remark}
	The naive dimension is fully determined by the combinatorics of $\A$, so we can define it on any abstract line combinatorics. Recall that in general, the naive dimension is not equal to the dimension of the moduli space, by Pappus' hexagon theorem.
\end{remark}

From~\cite[Remark~4.4]{DimIbaMac}, we have the following bound for the dimension of the moduli space.

\begin{proposition}\label{propo:naive_vs_dim}
	Let $\A$ be an arrangement. The following inequality holds:
	\[
		\naive \M(\A) \leq \dim_\CC \M(\A).
	\]
\end{proposition}

\section{Families of arrangements with connected moduli space}\label{sec:families}

In this section, we will study combinatorial classes of arrangements whose moduli space is connected. The first one is the notion of inductively connected arrangements introduced by Nazir and Yoshinaga in~\cite{NazirYoshinaga12}. It plays a crucial role in all the following.

\subsection{Inductively connected arrangement}\label{sec:inductively_connected}\mbox{}

The results of this section are reformulation and refinement of Nazir and Yoshinaga's~\cite{NazirYoshinaga12}. First, using the notion of the type of an arrangement introduced in Definition~\ref{def:type}, we can reformulate the definition of inductively connected arrangement.

\begin{definition}
	An arrangement $\A$ is \emph{inductively connected} if there exists an order $\omega$ on $\A$ such that $$\max(\tau(\A,\omega)) \leq 2.$$
\end{definition}

\begin{example}
	The arrangement $\A$ pictured in Figure~\ref{fig:ic} endowed with the order $\omega$ induced by the indices has type $\tau(\A,\omega)=(0,0,0,0,2,2,1)$. So it is an inductively connected arrangement.
\end{example}

\begin{figure}[!ht]
	\begin{tikzpicture}[scale=0.8]
		\def\Pout{4}
		\tikzstyle{droites}=[line width=1pt]
		\draw[droites, blue!60!green] (-\Pout,-1) -- (\Pout,-1) node[right] {$\ell_1$};
		\draw[droites, blue!60!green] (-\Pout,1) -- (\Pout,1) node[right] {$\ell_2$};
		\draw[droites, blue!60!green] (1,-\Pout) -- (1,\Pout) node[right] {$\ell_3$};
		\draw[droites, blue!60!green] (-1,-\Pout) -- (-1,\Pout) node[right] {$\ell_4$};
		\draw[droites] (1-\Pout,\Pout-1) -- (\Pout-1,1-\Pout) node[right] {$\ell_5$};
		\draw[droites] (1-\Pout,1-\Pout) -- (\Pout-1,\Pout-1) node[right] {$\ell_6$};
		\draw[droites] (2.5-\Pout,-\Pout) -- (\Pout-2,\Pout-1) node[above] {$\ell_7$};
	\end{tikzpicture}
	\caption{An inductively connected arrangement with type $\tau(\A,\omega)=(0,0,0,0,2,2,1)$.\label{fig:ic}}
\end{figure}
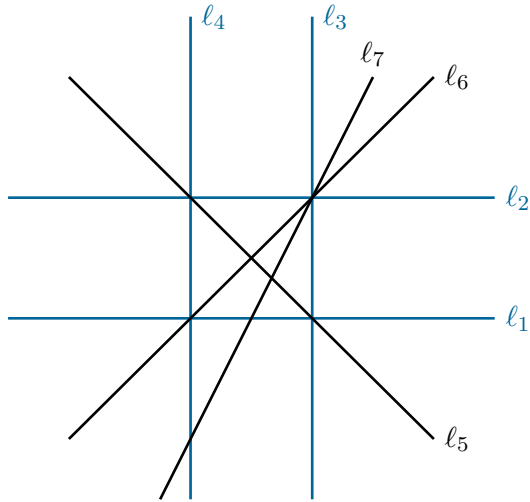

\begin{proposition}\label{propo:ic_square_basis}
	Let $\A\in\Arr_n$ be an inductively connected arrangement of $n\geq 4$ lines such that $\A$ is not in $\X(n)$ or $\ol{\X}(n)$. There exists an order $\omega$ on $\A$ verifying the following conditions:
	\begin{itemize}
		\item $\max(\tau(\A,\omega)) \leq 2$,
		\item the $4$th term $\A_4$ in the chain~\eqref{eqn:ordered_chain} is generic, i.e. $\A_4 \notin \X(4) \cup \ol{\X}(4)$.
	\end{itemize}
\end{proposition}

\begin{proof}
	Since $\A$ is inductively connected, there exists an order $\omega_0$ such that $\max(\tau(\A,\omega_0))\leq 2$. We assume that $\A_4$ is not generic, i.e. $\A_4$ is either $\X(4)$ or $\ol{\X}(4)$ (otherwise, we choose $\omega=\omega_0$), and also that $n\geq 5$. We denote by $P$ the unique multiple point of $\A_4$.

	Let $k$ be the smallest integer of $\{5,\dots,n\}$ such that $\A_k$ is formed by lines in the pencil centered at $P$ and two lines not passing through $P$. Note that $\A_k$ contains a generic subarrangement $\B$ of $4$ lines. Let $\omega_k$ be any order on $\A_k$ such that $\omega_k(\B) = \{1,2,3,4\}$, and let $\omega_0'$ be the order induced by $\omega_0$ on $\A\setminus\A_k$. By Remark~\ref{rmk:tau_partition_independence}, the order $\omega=\omega_k \oplus \omega_0'$ on $\A=\A_k\sqcup (\A\setminus\A_k)$ verifies the required conditions and the result holds.
\end{proof}

\begin{proposition}\label{propo:ic}
   If $\A$ is an inductively connected arrangement such that $\A\notin\X(n)\cup\ol{\X}(n)$, then $\M(\A)$ is isomorphic to a Zariski open subset of a complex space of dimension $\naive \M(\A)$. As a consequence, the moduli space $\M(\A)$ is connected.
\end{proposition}

\begin{proof}
	In the proof of~\cite[Lemma~3.2]{NazirYoshinaga12},
	the authors show that $\M(\A_{i+1})$ is isomorphic to either a $\CC\PP^1$-fibration or a trivial $\CC\PP^2$-fibration over $\M(\A_i)$, depending on the values of $\tau_i$.
	
	When $\tau_{i}=1$, the $\CC\PP^1$-fibration describes the choice of the new line $\ell_i$ in a pencil, which is supported by a singular point of $\A_{i-1}$. Given that there is already a line in $\A_{i-1}$ which belongs to this pencil and since $\ell_i$ should be different from this one, we can thus describe the choice of $\ell_i$ as a $\CC$-fibration over $\M(\A_{i-1})$.
	
	Similarly, when $\tau_i=0$ the $\CC\PP^2$-fibration describes the choice of a generic line $\ell_i$. Fix $P\in \Sing(\A_{i-1})$. Since $\ell_i$ should not pass through $P$ and thus it does not belong to the pencil supported on $P$, one can describe the choice of $\ell_i$ as a $\CC^2$-fibration.

	According to Proposition~\ref{propo:ic_square_basis}, we can assume that $\A_4$ is generic, and it can be completely fixed by the action of $\PGL_3$, i.e.~$\M(\A_4)=\set{\A_4}$. By induction, we obtain that $\M(\A)$ is a Zariski open subset of the complex space of dimension $\sum_{i=5}^n (2-\tau_i)$ and by Proposition~\ref{propo:sumtau_dnaive} the result holds.
\end{proof}

In fact, an explicit parametrization of $\M(\A)$ as an open Zariski subset of an affine complex space is provided in Section~\ref{sec:parametrization} for any arrangement $\A$.

\begin{corollary}\label{cor:ic_dim}
  If $\A\in\Arr_n$ is an inductively connected arrangement such that $\A\notin\X(n) \cup \ol{\X}(n)$, then
  \[
    \dim_\CC \M(\A) = \sum_{i=5}^n (2-\tau_i) = \naive \M(\A).
  \]
\end{corollary}

\begin{remark}
	In the same spirit as~\cite[Lemma~3.2]{NazirYoshinaga12}, let $(\A,\omega)$ and $(\A',\omega')$ be two ordered arrangements with $n$ and $n'$ lines respectively and such that $\tau_i(\A \sqcup \A',\omega\oplus\omega') \leq 2$ for all $i\in\{n+1,\dots,n+n'\}$. If $\M(\A)$ is irreducible, then $\M(\A\sqcup\A')$ is irreducible thus connected and it has dimension $$\dim_\CC \M(\A\sqcup\A') = \dim_\CC \M(\A) + \sum_{i=n+1}^{n+n'} (2-\tau_i(\A\sqcup\A',\omega\oplus\omega')).$$ Note that the reciprocal is not true in general, as it is shown in the following example.
\end{remark}

\begin{example}\label{ex:not_dense}
	In~\cite[\textsection~3.2]{GueViu:configurations} the authors present a line combinatorics $\C$ of $13$ lines with a moduli space with two irreducible components of dimension $1$. These components are geometrically characterized, e.g.~there are particular multiple points $P_1$, $P_2$, and $P_3$ which are collinear in the first component but they are not in the second one (see~\cite[Prop.~4.6]{GueViu:configurations}). Consider the combinatorics $\C'$ constructed on $\C$ by the addition of a line $\ell_{14}$ passing through $P_1$ and $P_2$ but avoiding $P_3$. By construction, one has that $\tau_{14}(\C') = 2$ and $\M(\C')$ is irreducible.
\end{example}

\subsection{Nice, simple and generalized simple arrangements}\label{sec:nice_simple}\mbox{}

In~\cite{JiaYau}, Jiang and Yau defined the combinatorial class of \emph{nice} arrangements, which was generalized by Wang and Yau in~\cite{WangYau} as the classes of \emph{simple} and \emph{generalized simple} arrangements. In the introduction of~\cite{NazirYoshinaga12}, Nazir and Yoshinaga ask about the relation between these combinatorial classes and the class of inductively connected arrangements. The purpose of this section is to prove that any generalized simple arrangement is actually inductively connected. First, let us recall the definition of generalized simple arrangements.

For any arrangement $\A$, we define $G_\A$ as the graph whose vertices $v_P$ are in one-to-one correspondence with the multiple points $P$ of $\A$ (i.e. the singular points whose multiplicity is at least $3$), and two vertices $v_P$ and $v_Q$ are joined by an edge if and only if the two corresponding points $P$ and $Q$ lie on the same line of $\A$. Associated with this graph, one define the following notions:
\begin{itemize}
  \item A \emph{star} centered in $v_P$, and denoted  $\St(v_P)$, is the subgraph of $G_\A$ generated by $v_P$ and any neighbor vertex of $v_P$. The open star $\overset{\circ}{\St}(v_P)$ is the complex composed of $v_P$ and any adjacent edge (without the neighbor vertices of $v_P$).

  \item A \emph{reduced circle} of $G_\A$ is a tuple of $(v_1,\dots,v_k)$ of vertices of $G_\A$, such that for all $i\in\{1,\dots,k\}$, one has that $(v_i,v_{i+1})$ is an edge of $G_\A$ and $v_{i-1}\neq v_{i+1}$ (the indices are considered modulo~$k$).

  \item A \emph{generalized free net} based on a tuple $(B_0,\dots,B_m)$ of reduced circles of $G_\A$ (where $B_0$ can also be a single vertex), is the maximal subgraph $\Net(B_0,\dots,B_m)$ of $G_\A$ whose vertices are at a distance of at most $1$ from at least one of the circles $B_i$, and such that
	\begin{enumerate}
		\renewcommand{\labelenumi}{(N\arabic{enumi})}
		\item for all $i\in\{0,\dots,m\}$, two vertices $v$ and $w$ of $B_i$ are connected by an edge if and only if they are adjacent in $B_i$,
		\item for all $i,j\in\{0,\dots,m-2\}$ and two vertices $v\in B_i$ and $w\in B_j$, it does not exist $z \in B_k$ with $k>\max(i,j)$ which verifies that $(v,z)$ and $(w,z)$ are both edges in $\Net(B_0,\dots,B_m)$,
		\item for all pair of vertices $v,w\in \Net(B_0,\dots,B_m)$, it does not exist $z \in G_\A \setminus \Net(B_0,\dots,B_m)$ which verifies that $(v,z)$ and $(w,z)$ are both edges in $G_\A$,
		\item for all $i\in\{1,\dots,n\}$, there exists a vertex $v\in B_i$ which is not connected by an edge to any reduced circle $B_j$ with $j<i$. The vertex $v$ is named the \emph{free vertex of the circle $B_i$}.
	\end{enumerate}
	The open net $\overset{\circ}{\Net}(B_0,\dots,B_m)$ is the complex formed by the vertices of the $B_i$ and all the edges of $\Net(B_0,\dots,B_m)$ (i.e.~all the end vertices of the net are removed).
\end{itemize}

\begin{definition}\label{defn:generalized_simple}
	An arrangement $\A$ is \emph{generalized simple} if there are stars $\St(v_1),\dots,\St(v_k)$ and generalized free nets $\Net_1,\dots,\Net_l$ which are pairwise disjoint in $G_\A$ and such that the graph
	\[
		G'_\A=G_\A\setminus \left\{  \overset{\circ}{\St}(v_1) , \dots, \overset{\circ}{\St}(v_k), \overset{\circ}{\Net}_1,\dots, \overset{\circ}{\Net}_l \right\}
	\]
	is a forest.
	When $l=0$, the arrangement $\A$ is \emph{nice}; and it is \emph{simple} when for all nets have $m=2$.
\end{definition}

\begin{theorem}\label{thm:simple_arrangement}
	If $\A$ is a generalized simple arrangement, then $\A$ is inductively connected.
\end{theorem}

In order to prove our result, the two following technical lemmas are required. They allow to remove all the lines which pass through a multiple point with a corresponding vertex of valence at most $2$ in the graph $G_\A$.

\begin{lemma}\label{lem:simple}
	Let $\A\in\Arr_n$ be an arrangement and $v_P$ a vertex of $G_\A$ of valence $m \leq 2$, and let $\A_P=\{L\in\A \mid P\in L\}$. There exists an order $\omega_P$ on $\A_P$ which verifies that for any order $\omega$ on $\A\setminus\A_P$, and for all $\ell \in \A_P$, we have that $\tau(\A,\omega \oplus \omega_P,\ell)\leq 2$.
\end{lemma}

\begin{proof}
	Let $\omega$ be a fixed order on $\A\setminus\A_P$. Note that if $\omega_P$ is an order of $\A_P$ then $\omega\oplus\omega_P$ is an order on $\A$. We study each case by the possible values of the valence:

	If $m=0$, then $v_P$ is an isolated vertex of $G_\A$. That is equivalent to say that any line of $\A_P$ has only $P$ as multiple point. So for any order $\omega_P$ on $\A_P$, one has $\tau(\A,\omega\oplus\omega_P,\ell) \leq 1$, for all $\ell\in\A_P$.

	If $m=1$, we denote by $\ell_P$ the line corresponding to the unique edge of $G_\A$ which ends at $v_P$. Remark that $\ell_P$ contains exactly two multiple points, since otherwise $m$ should be at least $2$. As in the previous case, any line of $\A_P \setminus \{\ell_P\}$ contains only $P$ as multiple point. Consider any order $\omega_P$ on $\A_P$ such that $\omega_P(\ell_P)=1$. For any line $\ell\in\A_P\setminus\{\ell_P\}$, we have that $\tau(\A,\omega\oplus\omega_P,\ell) \leq 1$. In $\A\setminus\{\ell\in\A_P \mid \ell\neq \ell_P\}$, the point $P$ is no longer a multiple point (neither a singular point). So, in the previous arrangement, $\ell_P$ contains a unique multiple point. Then we have $\tau(\A,\omega\oplus\omega_P,\ell_P)=1$.

	If $m=2$ and the two edges which join $v_P$ correspond to the same line $\ell_P$, then $\ell_P$ contains at most $3$ multiple points. We consider any order $\omega_P$ such that $\omega_P(\ell_P)=1$, then end of the proof is similar to the case $m=1$. The only difference is that in $\A\setminus\{\ell\in\A_P \mid \ell\neq \ell_P\}$, the line $\ell_P$ contains at most $2$ multiple points and thus $\tau(\A,\omega\oplus\omega_P,\ell_P)\leq 2$.

	If $m=2$ and the two edges which join $v_P$ correspond to two different lines $\ell_P$ and $\ell'_P$, then each line contains exactly two multiple points. So any line in $\A_P$ contains at most $2$ multiple points. This implies that for any order $\omega_P$ on $\A_P$, one has $\tau(\A,\omega\oplus\omega_P,\ell) \leq 2$, for all $\ell\in\A_P$.
\end{proof}

\begin{lemma}\label{lem:ic_induction}
	Let $\A$ be a line arrangement, let $v_P$ be a vertex of $G_\A$ with valence at most $2$ and let $\A_P=\{\ell \in\A \mid P\in\ell \}$. If $\A\setminus\A_P$ is inductively connected then so is $\A$.
\end{lemma}

\begin{proof}
	Assume that $\A\setminus\A_P$ inductively connected, then there exists an order $\omega_0$ on $\A\setminus\A_P$ such that $\tau(\A\setminus\A_P,\omega_0,\ell) \leq 2$, for all $\ell\in\A\setminus\A_P$. By Lemma~\ref{lem:simple}, there exists an order $\omega_P$ on $\A_P$ such that $\tau(\A,\omega_0\oplus\omega_P,\ell) \leq 2$, for all $\ell\in\A_P$. By Remark~\ref{rmk:tau_partition_independence}, we deduce that $\max(\tau(\A,\omega_0\oplus\omega_P))\leq 2$, and so that $\A$ is inductively connected.
\end{proof}

\begin{proof}[Proof of Theorem~\ref{thm:simple_arrangement}]
	For any multiple point $P$ of $\A$, we denote by $\A_P$ the arrangement $\{\ell\in\A\mid P\in\A\}$. Note that $G_{\A\setminus\A_P}$ is a (non necessarily strict) subgraph of $G_\A\setminus \overset{\circ}{\St}(v_P)$. The strategy of the proof is to remove one-by-one all the vertices of $G_\A$ in such a way that, at each step, the removed vertex has valence at most $2$.

	\medskip

	\noindent {\sc Step 1}. The vertices of $G'_\A$. %
	
	By hypothesis, $G'_\A$ is a forest. Let $v_P$ be an end-point of $G'_\A$, that is there is at most one edge in $G'_\A$ which meets $v_P$. Then, we know that the stars and the generalized net are disjoint. Furthermore, condition~(N3) implies that there are no two edges in the same generalized net that meet $v_P$. This implies that there is at most $1$ edge in the union of the stars and of the generalized nets which meet $v_P$. So the valence of $v_P$ is at most $2$. So by Lemma~\ref{lem:ic_induction}, if $\A\setminus\A_P$ is inductively connected then so is $\A$.\\
	As noted at the beginning of the proof, $G_{\A\setminus\A_P}$ is a subgraph of $G_\A \setminus \overset{\circ}{\St}(v_P)$. The remaining part of $G'_\A$ in $G_{\A\setminus\A_P}$ is also a forest (as a subgraph of a forest) and then we can apply the same argument on a new end-point. By successive applications of this argument, we obtain that if $\B=\A\setminus \bigcup_{v_P\in G'_\A} \A_P$ is inductively connected then so is $\A$.

	\medskip

	\noindent {\sc Step 2}. The vertices of $\St(v_P)$.

	The graph $G_\B$ is contained in the disjoint union of the stars and of the generalized nets. Remark that the union of the stars is also a forest. So we can apply exactly the same argument as in~{\sc step~1} to remove all the vertices included in any star. Thus, we obtain that if $\D=\B\setminus \bigcup_{v_P\in \cup \St(v_i)} \A_P$ is inductively connected then so is $\B$.

	\medskip

	\noindent {\sc Step 3}. The vertices of $\Net_i$.

	The graph $G_\D$ is contained in the disjoint union of the generalized nets. By definition, we have that $\Net_l=\Net(B_0,\dots,B_m)$ for a fixed $m$. Remark that all the vertices outside the reduced circles $B_i$ are not in $G_\D$ since they have been removed at Step~1. Let $v_P$ be the free vertex of $B_m$, and let $v_Q$ be one of its two neighbors in $B_m$. According to conditions~(N1) and~(N4), the valence of $v_f$ is exactly $2$. By Lemma~\ref{lem:ic_induction}, if $\D\setminus\A_P$ is inductively connected then so is $\D$. It follows from conditions~(N1) and~(N2), that in $G_{\D\setminus\A_P}$ the vertex $v_Q$ as valence at most $2$. So we can inductively apply this argument on all the vertices of $B_m$, and then on each reduced circle $B_i$ of $\Net_l$. At the end, we obtain that if $\D\setminus\bigcup_{i=0}^m\bigcup_{v_P\in B_i} \A_P$ is inductively connected then so is $\D$.

	Applying the latter on each generalized net $\Net(B^j_0,\dots,B^j_{m_j})$, we obtain that if the arrangement $\E = \D\setminus \bigcup_{j=0}^l\bigcup_{i=0}^{m_j}\bigcup_{v_P\in B^j_i} \A_P$ is inductively connected, then so is $\D$.

	\medskip

	As a consequence of the previous steps, the obtained graph $G_\E$ is empty. Thus, the arrangement $\E$ is either empty or generic. In both cases, we have that $\E$ is inductively connected. According to the previous steps, this implies that $\A$ is also inductively connected.
\end{proof}

The fact that the diffeomorphism type of the complement of both nice and generalized simple arrangements is combinatorially determined is well-known and forms the main theorems of~\cite{JiaYau} and~\cite{WangYau}, respectively. This can also be deduced from Theorem~\ref{thm:simple_arrangement}, Proposition~\ref{propo:ic} and Randell's Lattice–Isotopy Theorem~\cite{Ran:lattice_isotopy}.

\subsection{Inductively rigid arrangements}\label{sec:inductively_rigid}\mbox{}

Following a similar approach as for inductively connected arrangements, we introduce another combinatorial class of arrangements with connected moduli space.

\begin{definition}\label{def:inductively_rigid}
	An arrangement $\A\in\Arr_n$ is \emph{inductively rigid} if there exists an order $\omega$ such that for any arrangement $\A_i$ in the ascending chain~\eqref{eqn:ordered_chain}, we have $\M(\A_i)=\{\A_i\}$.
\end{definition}

Note that if an arrangement $\A$ is inductively rigid, then any $\A_i$ in the induced ordered chain~\eqref{eqn:ordered_chain} is also inductively rigid. %
As a consequence, any arrangement $\A$ with $|\A|\leq3$ is inductively rigid. In addition, no arrangement in the classes $\X(n)$ and $\ol{\X}(n+1)$ is inductively rigid for $n\geq4$, by Remark~\ref{rmk:moduli_pencils}.

\begin{proposition}\label{propo:inductively_rigid}
	An arrangement $\A\in\Arr_n$ is inductively rigid if and only if either $n\leq 4$ and $\A\notin\X(4)$, or $n\geq 5$ and there exists an order $\omega$ on $\A$ such that:
	\begin{enumerate}[label=(A\arabic*), ref=(A\arabic*)]
		\item \label{A1} $\A_4$ is generic, i.e. for all $i\in\{1,2,3,4\}$, one has $\tau_i(\A,\omega) = 0$,
		\item \label{A2} for all $i\in\{5,\dots,n\}$, one has $\tau_i(\A,\omega) \geq 2$.
	\end{enumerate}
\end{proposition}

\begin{proof}
	Any arrangement of at most three lines is inductively rigid. Whenever $n = 4$, either $\A\in\X(4)$ or $\A$ is fixed by the $\PGL_3$-action and thus $\M(\A)=\set{\A}$. The first case is not inductively rigid while the second one is. Assume that $n\geq 5$, one has that:
	\begin{itemize}
	  \item If for any order $\omega$ the condition \ref{A1} does not hold, then $\A\in\X(n)\cup\ol{\X}(n)$ with $n\geq5$ and thus, due to Remark~\ref{rmk:moduli_pencils}, $\A$ is not inductively rigid.
	  
	  \item If for any order $\omega$, the condition \ref{A2} does not hold, then there exists $i\geq 5$ such that $\tau_i(\A,\omega)<2$. This implies that $\dim_\CC \M(\A_i) > 0$ and thus $\A$ is not inductively rigid.
	  
	  \item If there exists an order $\omega$ over $\A$ such that \ref{A1} and \ref{A2} hold, then by a recursive application of~\cite[Lemma~3.2]{NazirYoshinaga12}, we deduce that $\M(\A_i)=\{\A_i\}$ for all $i\in\{5,\dots,n\}$. Then, $\A$ is inductively rigid.
	\end{itemize}
	Since we have covered all possible cases, we obtain the stated equivalence.
\end{proof}

As an important consequence of the previous result, we conclude that the class of inductively rigid arrangements is combinatorial. The former class is not contained in the class of inductively connected arrangements, as it is shown in the following example.

\begin{example}
	For $n\geq 1$, consider the family of arrangements $\mathcal{Y}_n$ defined by the $4(n+1)$ lines: $z=0$, $x - k z = 0$, $y - k z = 0$, for $k\in\{0,\dots,n\}$ and $x + y - k z = 0$ for $k\in\{0,\dots,2n\}$, see Figure~\ref{fig:ir_arrangement}. It is left to the reader to verify that the arrangements $\mathcal{Y}_n$ are inductively rigid but not inductively connected for $n\geq2$.
\end{example}

\begin{figure}[h]
	\begin{tikzpicture}[scale=0.75]
		\tikzstyle{droite}=[line width=1pt]
		\draw[droite] (-3,0) -- (5,0);
		\draw[droite] (-3,1) -- (5,1);
		\draw[droite] (-3,2) -- (5,2);

		\draw[droite] (0,-3) -- (0,5);
		\draw[droite] (1,-3) -- (1,5);
		\draw[droite] (2,-3) -- (2,5);

		\draw[droite] (-3,3) -- (3,-3);
		\draw[droite] (-3,4) -- (4,-3);
		\draw[droite] (-3,5) -- (5,-3);
		\draw[droite] (-2,5) -- (5,-2);
		\draw[droite] (-1,5) -- (5,-1);

		\draw[droite] (5.5,4) to[out=100,in=-10] node [above right, near start] {$\infty$} (4,5.5);
	\end{tikzpicture}
	\caption{The arrangement $\mathcal{Y}_2$.\label{fig:ir_arrangement}}
\end{figure}
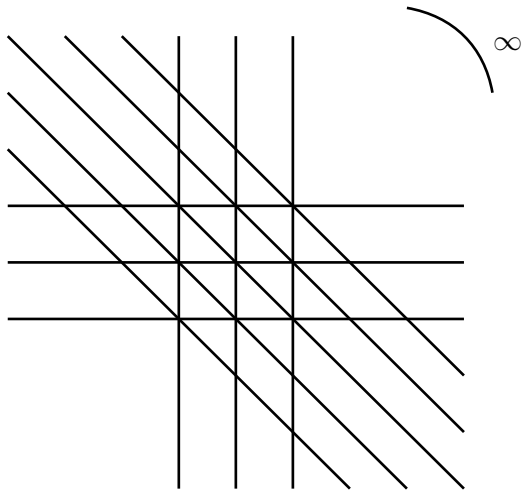

\subsection{Arrangements with a rigid pencil form}\label{sec:rigid_pencil_form}\mbox{}

An arrangement $\A$ is $C_k$ if $k$ is the minimal integer such that there exists a subarrangement $\D_k\subset \A$ of $k$ lines such that $\Sing_{\geq3}(\A)$ is contained in $\D_k$. Nazir and Yoshinaga proved that if an arrangement is $C_0$, $C_1$ or $C_2$ then it is inductively connected, and so that its moduli space is connected, see~\cite[Theorem.~3.11]{NazirYoshinaga12}. They also introduced the combinatorial class $C_3$ of simple\footnote{The term simple used here is not related to the one of Section~\ref{sec:nice_simple}.} type.

\begin{definition}\label{def:C3_simple}
	An arrangement $\A$ of class $C_3$ is of \emph{simple type} if one of following conditions holds:
	\begin{enumerate}
		\renewcommand{\labelenumi}{(\roman{enumi})}
		\item the three lines of $\D_3$ are in generic position, and one of them contains a unique multiple point,
		\item the three lines of $\D_3$ are concurrent.
	\end{enumerate}
\end{definition}

\begin{theorem}[{\cite[Thm.~3.15]{NazirYoshinaga12}}]
	If $\A$ is a $C_3$ arrangement of simple type then its moduli space $\M(\A)$ is connected.
\end{theorem}

The main argument to prove that $C_3$ arrangement of simple type with concurrent lines is irreducible use the fact that one can fix the three lines using the action of $\PGL_3(\CC)$. Using the combinatorial notion of inductively rigid arrangements, one can define a larger class of arrangements with a connected moduli space.

\begin{definition}\label{def:rigid_pencil}
	An arrangement $\A$ has a \emph{rigid pencil form} if it contains an inductively rigid subarrangement $\A'$ with a singular point $P_0$ such that for any multiple point $P\in\Sing_{\geq 3}(\A)$ one of the following conditions holds:
	\begin{enumerate}[label=(RP\arabic*), ref=(RP\arabic*), leftmargin=4em]
		\item\label{RP1} $P$ is a singular point of $\A'$, 
		\item\label{RP2} the line $(P,P_0)$ is in $\A'$.
	\end{enumerate}
\end{definition}

\begin{proposition}
  Any $C_3$ arrangement of simple type has a rigid pencil form.
\end{proposition}

\begin{proof}
  Let $\A$ be a $C_3$ arrangement of simple type. Assume that $\A$ verifies condition~(i) of Definition~\ref{def:C3_simple}. Let $\D_3=\{\ell,\ell',\ell''\}$ be a generic subarrangement that contains all the multiple points of $\A$, and such that $\ell''$ contains a unique multiple point $Q_0$. There exist two lines $\ell_{Q_0},\ell'_{Q_0}\in\A$ passing through $Q_0$ which do not contained $\ell\cap\ell'$. The arrangement $\A'=\{\ell,\ell',\ell_{Q_0},\ell'_{Q_0}\}$ is inductively rigid. The multiple point $Q_0$ is contained in $\ell_{Q_0}$ and $\ell'_{Q_0}$, and all the other multiple points of $\A$ are contained in $\ell \cup \ell'$. So $\A$ has a rigid pencil form.

  Assume that $\A$ verifies condition~(ii) of Definition~\ref{def:C3_simple}. Then any multiple point is in the $3$-pencil $\D_3$. Since $\D_3$ is in $\X(3)$, then it is inductively rigid. So $\A$ has a rigid pencil form.
\end{proof}

\begin{theorem}\label{thm:C3_generalization}
	If $\A$ has a rigid pencil form then its moduli space $\M(\A)$ is connected.
\end{theorem}

\begin{proof}
	The idea of the proof is similar to the one of~\cite[Proof~of~Theorem~3.15, case~(ii)]{NazirYoshinaga12}. One can assume first that $\A'$ is maximal, in the sense that for all $\ell\in\A\setminus\A'$ one has $|\ell\cap\Sing(\A')| < 2$. Up to a projective transformation, one can also assume that $P_0$ has coordinates $(0:1:0)$ and that any line $\ell_i\in\A$ which pass through $P_0$ is defined by the equation $x - \lambda_i z = 0$, for a fixed $\lambda_i\in\CC$. %
	Denote by $P_i$ the multiple points of $\A$ which are not singular points of $\A'$. By Condition~\ref{RP2}, each $P_i$ has coordinates $(\lambda_i:\alpha_i:1)$, for some $\alpha_i\in\CC$. The moduli space is then defined by some collinearity relations between the points $P_i$ and the singular points of $\A'$. More precisely, there are two types:
	\begin{itemize}
		\item The collinearity between three points  $(P_i,P_j,P_k)$ given by the linear equation on $\alpha_i$,  $\alpha_j$ and~$\alpha_k$:
		\[
			\alpha_i (\lambda_k - \lambda_j) + \alpha_j (\lambda_i - \lambda_k) + \alpha_k (\lambda_j - \lambda_k) = 0.
		\]
		\item The collinearity between a point $Q\in\Sing(\A')$ and two points $P_i$ and $P_j$
		given by the affine equation on $\alpha_i$ and~$\alpha_j$:
		\[
			\alpha_i (r_Q - \lambda_j t_Q) + \alpha_j (\lambda_i t_Q - r_Q) + s (\lambda_i - \lambda_j) = 0.
		\]
		where $Q=(r_Q:s_Q:t_Q)\in\CC\PP^2$.
	\end{itemize}
	Due to the assumption that $\A'$ is maximal, there is no other type of relations. %
	Since $\M(\A')=\set{\A'}$ by hypothesis, the moduli space $\M(\A)$ is described by the previous equations on the $\alpha_i$'s with fixed values $\set{\lambda_i}_{\ell_i\in\A'}$ and $\set{(r_Q:s_Q:t_Q)}_{Q\in\Sing(\A')}$. %
	Thus, $\M(\A)$ is an open Zariski subset of the kernel of an affine map defined by the previous equations. As a consequence, it is irreducible and then connected.
\end{proof}

\begin{example}
	Consider the arrangement $\A\in\Arr_{10}$ pictured in Figure~\ref{fig:k_pencil} and defined by the equations:
	\[
		\begin{array}{cccc}
			\ell_1: x = 0, &\quad \ell_2: x - z = 0, &\quad \ell_3: y = 0, &\quad \ell_4: y - z = 0, \\
			\ell_5: z = 0, &\quad \ell_6: -x + y = 0, &\quad \ell_7: x + y - z = 0, &\quad \ell_8: -2x + 4y - z = 0, \\
			\ell_9: 2x - 3y + z = 0, &\quad \ell_{10}: -4x + 6y = 0.
		\end{array}
	\]
	Its combinatorics is given by:
	\[
		\left\lbrace%
		\begin{array}{c}
			\left\{\ell_1,\ell_2,\ell_5\right\},\left\{\ell_1,\ell_3,\ell_6,\ell_{10}\right\},\left\{\ell_1,\ell_4,\ell_7\right\},\left\{\ell_1,\ell_8\right\},\left\{\ell_1,\ell_9\right\},\left\{\ell_2,\ell_3,\ell_7\right\}\\[0.3em]
			\left\{\ell_2,\ell_4,\ell_6,\ell_9\right\},\left\{\ell_2,\ell_8\right\},\left\{\ell_2,\ell_{10}\right\}, \left\{\ell_3,\ell_4,\ell_5\right\},\left\{\ell_3,\ell_8,\ell_9\right\},\left\{\ell_4,\ell_8,\ell_{10}\right\} \\[0.3em]
			\left\{\ell_5,\ell_6\right\},\left\{\ell_5,\ell_7\right\},\left\{\ell_5,\ell_8\right\},\left\{\ell_5,\ell_9,\ell_{10}\right\},\left\{\ell_6,\ell_7,\ell_8\right\},\left\{\ell_7,\ell_9\right\},\left\{\ell_7,\ell_{10}\right\}
		\end{array}
		\right\rbrace.
	\]
	Let $\A'=\{\ell_1,\dots,\ell_7\}$ be the subarrangement of $\A$ with combinatorics:
	\[
		\left\{\left\{\ell_1,\ell_2,\ell_5\right\},\left\{\ell_1,\ell_3,\ell_6\right\},\left\{\ell_1,\ell_4,\ell_7\right\},\left\{\ell_2,\ell_3,\ell_7\right\},\left\{\ell_2,\ell_4,\ell_6\right\},\left\{\ell_3,\ell_4,\ell_5\right\},\left\{\ell_5,\ell_6\right\},\left\{\ell_5,\ell_7\right\},\left\{\ell_6,\ell_7\right\}\right\}.
	\]
	Once again, if we consider the order $\omega$ induced by the indices, we have $\tau(\A',\omega)=(0,0,0,0,2,2,2)$ and thus  $\A'$ is an inductively rigid arrangement. Furthermore, all multiple points of $\A$ are contained in the three concurrent lines $\ell_3$, $\ell_4$ and $\ell_5$, except for the point $\{\ell_6,\ell_7,\ell_8\}$ which contains the lines $\ell_6$ and $\ell_7$ of $\A'$. Thus $\A$ has a rigid pencil form. By Theorem~\ref{thm:C3_generalization}, we deduce that $\M(\A)$ is connected.

	It is worth noticing that all multiple points of $\A$ are also contained in the $4$-pencil $\{\ell_1, \ell_3, \ell_6, \ell_{10}\}$ or $\{\ell_2, \ell_4, \ell_6, \ell_9\}$. This implies that $\A$ is a $C_4$ arrangement. Nevertheless, any of these two $4$-pencils is contained in an inductively rigid subarrangement of $\A$. So $\A$ does not admit a rigid pencil form.

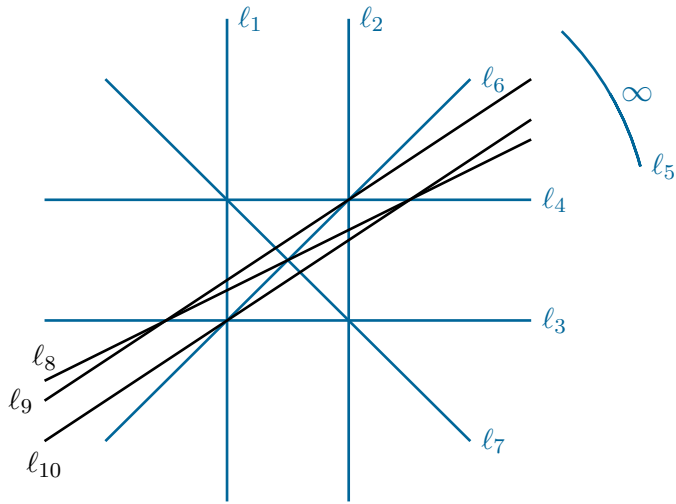
\begin{figure}[h]
	\begin{tikzpicture}[scale=0.8]
		\def\Pout{4}
		\tikzstyle{droites}=[line width=1pt]
		\draw[droites, blue!60!green] (-\Pout,-1) -- (\Pout,-1) node[right] {$\ell_3$};
		\draw[droites, blue!60!green] (-\Pout,1) -- (\Pout,1) node[right] {$\ell_4$};

		\draw[droites, blue!60!green] (1,-\Pout) -- (1,\Pout) node[right] {$\ell_2$};
		\draw[droites, blue!60!green] (-1,-\Pout) -- (-1,\Pout) node[right] {$\ell_1$};

		\draw[droites, blue!60!green] ([shift=(15:\Pout+2)]0,0) node[right] {$\ell_5$}  arc (15:30:\Pout+1) node[right] {\large $\infty$} ([shift=(15:\Pout+2)]0,0) arc (15:45:\Pout+1);

		\draw[droites, blue!60!green] (1-\Pout,\Pout-1) -- (\Pout-1,1-\Pout) node[right] {$\ell_7$};
		\draw[droites, blue!60!green] (1-\Pout,1-\Pout) -- (\Pout-1,\Pout-1) node[right] {$\ell_6$};

		\draw[droites,black] (4,3) -- (-4,-2.33) node[left] {$\ell_9$};
		\draw[droites,black] (4,2.33) -- (-4,-3) node[below] {$\ell_{10}$};
		\draw[droites,black] (\Pout,\Pout-2) -- (-\Pout,-2) node[above] {$\ell_{8}$};
	\end{tikzpicture}
	\caption{An arrangement with rigid pencil form.\label{fig:k_pencil}}
\end{figure}
\end{example}

\begin{remark}
	The statement of Theorem~\ref{thm:C3_generalization} still holds if the subarrangement $\A'$ verifies that $\M(\A')=\{\A'\}$ instead of being inductively rigid. Nevertheless, one should notice that this is a geometric but not combinatorial condition.
\end{remark}

\section{On the number of connected components of $\M(\A)$}\label{sec:irred_components}

The objective of this section is to introduce a strategy to inductively build an upper bound on the number of connected components of $\M(\A)$ by means of certain ``tower of modifications'' in the combinatorics.

\subsection{Perturbations of a combinatorics}\label{sec:perturbation}\mbox{}

\begin{definition}\label{def:elementary_perturabtion}
	Let $\C=(\L,\P)$ be an abstract line combinatorics. Let $P_0\in\P$ be a fixed multiple point of $\C$ and let $\ell \in P_0$. An \emph{elementary perturbation} of $\C$ at $(\ell,P_0)$ is an abstract combinatorics $\widetilde{\C}=(\L,\widetilde{\P})$ such that:
	\begin{enumerate}
		\item for all $P \in \P$, $P\neq P_0$, one has $P\in\widetilde{\P}$,
		\item $\widetilde{P}_0 = P_0 \setminus \{\ell\}$ is in $\widetilde{\P}$.
	\end{enumerate}
	Such a relation is denoted $\widetilde{\C} \prec \C$. For any multiple point $P\in\P$, we call the \emph{parent of $P$} in $\widetilde{\C}$ the unique element $\widetilde{P}\in\widetilde{\P}$ such that either $\widetilde{P}=P$ or $\widetilde{P}=\widetilde{P}_0$.
\end{definition}

\begin{example}
	Let $\C=\{\{\ell_1,\ell_2,\ell_3,\ell_4\},\{\ell_1,\ell_5\},\{\ell_2,\ell_5\},\{\ell_3,\ell_5\},\{\ell_4,\ell_5\} \}$ be the combinatorics of an arrangement in~$\ol{\X}(5)$. The elementary perturbation of $\C$ at $(P_0,\ell_4)$, for $P_0=\{\ell_1,\ell_2,\ell_3,\ell_4\}$, is:
	\[
		\{\{\ell_1,\ell_2,\ell_3\},\{\ell_1,\ell_4\},\{\ell_1,\ell_5\},\{\ell_2,\ell_4\},\{\ell_2,\ell_5\},\{\ell_3,\ell_5\},\{\ell_4,\ell_5\},\{\ell_4,\ell_5\} \},
	\]
	and the parent of $P_0=\{\ell_1,\ell_2,\ell_3,\ell_4\}$ is $\{\ell_1,\ell_2,\ell_3\}$.
\end{example}

\begin{lemma}\label{lem:elementary_perturabtion}
	If $\widetilde{C}$ is an elementary perturbation of $\C$ then $\sum_{P\in\widetilde{\P}} (|P|-2) = \sum_{P\in\P} (|P|-2) - 1$.
\end{lemma}

\begin{proof}
	Assume that $\widetilde{\C}$ is an elementary perturbation of $\C$ at $(\ell,P_0)$. By Definition~\ref{def:elementary_perturabtion}-(1), all the terms in the right-sided sum, except for that depending on $P_0$, are also in the left-sided sum. Since $\widetilde{P}_0$ is the parent of $P_0$, one has $(|\widetilde{P}_0|-2) = (|P_0|-2) + 1$. Finally, notice that all remaining terms of the left-sided sum correspond to double points of $\widetilde{\C}$ and thus they do not contribute to the sum.
\end{proof}

Note that, according to Definition~\ref{def:naive}, the previous lemma implies that
\[
	1 + \naive \M(\C) = \naive \M(\widetilde{\C}).
\]

\begin{definition}\label{def:m_perturbation}
	A \emph{$m$-perturbation} of $\C=(\L,\P)$ is a sequence of $m$ elementary perturbations such that
	\[
		\C_0 \prec \C_1 \prec \cdots \prec \C_m = \C,
	\]
	and $\C_0$ is inductively connected with associated order $\omega_0$. Since $\C_0$ and $\C$ have the same set of lines, then $\omega_0$ is also an order on $\C$, and it will be called the \emph{perturbation order}. For any multiple point $P$ in $\P_i$, there is a unique chain of parents coming from each elementary perturbation:
	\[   
	  \widetilde{P}=P_0 \subset P_1 \subset \cdots \subset P_i=P.
	\]
	The multiple point $\widetilde{P}$ in $\C_0$ is called the \emph{ancestor of $P$}.
\end{definition}

For brevity sake, a $m$-perturbation $\C_0 \prec \C_1 \prec \cdots \prec \C_m = \C$ will be denoted by $\C_0 \pprec\C$. %
Any combinatorics can be reduce to the generic combinatorics by enough successive elementary perturbations. Since the generic combinatorics is inductively connected then any combinatorics admits a $m$-perturbation for some $m$.

\subsection{Parametrization of the moduli space.}\label{sec:parametrization}\mbox{}

The purpose of this section is to give a constructible proof of the following result.

\begin{theorem}\label{thm:moduli_embedding}
	Let $\A\in\Arr_n$ be an arrangement, with $\A\notin \X(n)\cup\ol{\X}(n)$, which admits a $m$-perturbation $\C_0 \pprec \C(\A)$, and denote $d_0=\naive \M(\C_0)$. There exists an open Zariski subset $W_0$ of the affine space of dimension $d_0$ and $m$ polynomials $\Delta_1,\ldots,\Delta_m$ such that
	\begin{equation}\label{eq:moduli_descritpion}
		\M(\A) \simeq \left\{ (v_1,\dots,v_{d_0})\in W_0 \mid \Delta_1 = \dots = \Delta_m = 0 \right\}.
	\end{equation}
\end{theorem}

Remark that we may have that some of the $\Delta_i$'s are always trivial. It is for example the case in Pappus arrangement. From the previous result, we can deduce the following corollary. The left-hand inequality corresponds to the one of  Proposition~\ref{propo:naive_vs_dim}.

\begin{corollary}\label{cor:dimension_bound}
	Let $\A\in\Arr_n$ be a line arrangement. If $\C(\A)$ admits a $m$-perturbation then
	\[
		0 \leq \dim_\CC \M(\A) - \naive \M(\A) \leq m.
	\]
\end{corollary}

\begin{proof}
  According to Theorem~\ref{thm:moduli_embedding}, we can describe $\M(\A)$ as the intersection of at most $m$ hypersurfaces in some non-empty open Zariski subset of the affine space of dimension $\naive \M(\C_0)$. So, using classical arguments on the codimension of the intersection of hypersurfaces, we deduce that $\naive \M(\C_0) \geq \dim_\CC \M(\A) \geq \naive \M(\C_0) - m$.

  By successive applications of Lemma~\ref{lem:elementary_perturabtion}, we have that $\sum_{P\in\widetilde{\P}_0}(|P|-2) = \sum_{P\in\P}(|P|-2) - m$. This can be reformulated as $\naive \M(\C_0) = \naive \M(\A) + m$.  Then, we substitute $\naive \M(\C_0)$ by $\naive \M(\A) + m$ in the previous inequalities, and we obtain:
  \[
	\naive \M(\A) \leq \dim_\CC \M(\A) \leq \naive \M(\A) + m. \qedhere
  \]
\end{proof}

\medskip

The strategy of the proof of Theorem~\ref{thm:moduli_embedding} is the following. The goal is to express the moduli space $\M(\A)$ described in \eqref{eqn:moduli_space} from the following data:
\begin{itemize}
  \item An open Zariski subset $W_0$ of an affine linear space, codifying the open conditions in~\eqref{eqn:moduli_space}.
  
  \item A map $\Psi$ over $W_0$ which parametrizes the arrangements in $\Arr_n$ verifying the Zariski-closed conditions of $\M(\C_0)$, e.g~those of type $\Delta_{i,j,k}=0$ with $\set{i,j,k}\subset \widetilde{P}$ for a multiple point $\widetilde{P}$ in $\C_0$.
  
  \item A list $\Delta_1,\ldots,\Delta_m$ of $m$ polynomials, i.e.~the remaining Zariski-closed conditions coming from multiple points $P$ in $\C(\A)$ which are not in $\C_0$, where each polynomial represents a step of the $m$-perturbation $C_0\pprec\C(\A)$.
\end{itemize}
We start by constructing an expression of the map $\Psi$, then we define the polynomials $\Delta_1,\ldots,\Delta_m$ and finally the existence of $W_0$ is discussed.

\begin{remark}\label{remark:set_conditions}
  Let $P=\set{\ell_{i_1},\ldots,\ell_{i_m}}$ be a multiple point in $\C$. The set of Zariski-closed conditions $\set{\Delta_{i,j,k}=0 \mid \set{i,j,k}\subset P}$ can be reduced to $\set{\Delta_{i_1,i_2,i_k}=0 \mid k=3,\ldots,m}$.
\end{remark}

Let $\bv=(v_1,\dots,v_{d_0})$ be a system of coordinates in $V_0=\CC^{d_0}$. For each line $\ell_i$, we associate three polynomials $a_i, b_i, c_i\in\CC[v_1,\dots,v_{d_0}]$ such that
\[
  \ell_i: a_i(\bv)x+b_i(\bv)y+c_i(\bv)z=0.
\]
In this way, one can define a map called \emph{parametrization} of $\M(\A)$:
\begin{equation}\label{eq:Psi}
  \Psi: \bv\in W_0 \longmapsto
	\left(
	\left(
		\begin{array}{c}
			a_1(\bv) \\ b_1(\bv) \\ c_1(\bv)
		\end{array}
	\right)
	, \ldots ,
	\left(
		\begin{array}{c}
			a_n(\bv) \\ b_n(\bv) \\ c_n(\bv)
		\end{array}
	\right)
	\right)
	\in(\CC^3)^{n}.
\end{equation}
In addition to $\Psi$, one can construct another map $\Phi$ that parameterize the multiple points\footnote{In the following, it may happen that we define the map $\Phi$ for a double point of $\C(\A)$. This allows to lighten the explanation and it has no incidence on the result.} of the elements in $\M(\A)$ as follows. Let $P$ be a multiple point in the combinatorics $\C(\A)$, and let $\widetilde{P}$ be the ancestor of $P$ in $\C_0$. Take $\ell_i$ and $\ell_j$ the two lines in $\widetilde{P}$ which are minimal with respect to the order $\omega_0$.
We define
\begin{equation}\label{eq:intersection_points}
	\Phi_P = (b_i c_j - b_j c_i , a_j c_i - a_i c_j , a_i b_j - a_j b_i) \in \CC[v_1,\dots,v_{d_0}]^3,
\end{equation}
If an arrangement $\A_0\in\M(\A)$ is given by $\Psi(\bv)$ then for any multiple point $P$ of $\C(\A)$, the vector $\Phi_P(\bv)\in\CC^3$ express the homogeneous coordinates of $P$ in $\CC\PP^2$.

\medskip

Let us describe in detail how the polynomials $a_i$, $b_i$ and $c_i$ are inductively constructed. Since $\A\notin \X(n)\cup \ol{\X}(n)$, one can follow similar arguments as in Proposition~\ref{propo:ic_square_basis}, and assume that the perturbation order $\omega_0$ is such that the lines $\ell_1$, $\ell_2$, $\ell_3$ and $\ell_4$ are in generic position in both $\C_0$ and $\C(\A)$. Using the action of $\PGL_3(\CC)$, we fix them as $x=0$, $x-z=0$, $y=0$ and $y-z=0$, respectively. In other words, for all $\bv\in V_0$, we define
\begin{equation}\label{eq:square_lines}
		\Psi(\bv)_1 = (1,0,0),\quad
		\Psi(\bv)_2 = (1,0,-1),\quad
		\Psi(\bv)_3 = (0,1,0),\quad
		\Psi(\bv)_4 = (0,1,-1).
\end{equation}
It follows from equation~(\ref{eq:intersection_points}) that the parametrization $\Phi$ of the singular points of these four lines are given for any $\bv\in V_0$ by:
\begin{equation}\label{eq:square_pts}
	\begin{array}{ccccc}
		\Phi_{\langle\ell_1 , \ell_2\rangle}(\bv) = (0,1,0), & \quad &
		\Phi_{\langle\ell_1 , \ell_3\rangle}(\bv) = (0,0,1), & \quad &
		\Phi_{\langle\ell_1 , \ell_4\rangle}(\bv) = (0,1,1), \\
		\Phi_{\langle\ell_2 , \ell_3\rangle}(\bv) = (1,0,1), & \quad &
		\Phi_{\langle\ell_2 , \ell_4\rangle}(\bv) = (1,1,1), & \quad &
		\Phi_{\langle\ell_3 , \ell_4\rangle}(\bv) = (1,0,0).
	\end{array}
\end{equation}
where $\langle\ell_{i_1},\ldots,\ell_{i_k}\rangle$ with $k\geq2$ is the unique point $P$ in $\C(\A)$ verifying that $\set{\ell_{i_1},\ldots,\ell_{i_k}}\subset P$, if it exists.

The induction process goes as follows. Assume that the maps $\Psi$ and $\Phi$ parametrize $\A_{i-1}$ in the chain~\eqref{eqn:ordered_chain} with respect to the perturbation order $\omega_0$. Note that the number of parameters used in this parametrization is $d_i=\sum_{j=5}^{i-1} \big(2 - \tau_j(\C_0,\omega_0))$.
The next step is then to extend these parametrization maps to $\A_i$ . They are determined by the values $\tau_i(\C_0,\omega_0)$:

\begin{itemize}
	\item If $\tau_i(\C_0,\omega_0) = 0$, i.e. the line $\ell_i$ is generic in $(\C_0)_{i}$. %
	Since $\ell_i\not\in\gen{\ell_1,\ell_2}$ or $\ell_i\not\in\gen{\ell_1,\ell_3}$ in $\C(\A)$, we can parametrize $\ell_i$ using only two complex parameters $v_{d_i}$ and $v_{d_{i}+1}$,
	\begin{equation}\label{eq:generic_line}
		\Psi(\bv)_i = \begin{cases}      
		                (v_{d_i},1,v_{d_i+1}) & \text{if $\ell_i\not\in\gen{\ell_1,\ell_2}$},\\
		                (1,v_{d_i},v_{d_i+1}) & \text{if $\ell_i\not\in\gen{\ell_1,\ell_3}$}.
		              \end{cases}
	\end{equation}

	\item If $\tau_i(\C_0,\omega_0) = 1$, i.e. the line $\ell_i$ passes through a unique multiple point $P_0$ in $(\C_0)_{i}$. Take $\ell_j$ and $\ell_k$ the two lines in $P_0$ which are minimal with respect to the order $\omega_0$. One define:	
	\begin{equation}\label{eq:pencil_line}
		\Psi(\bv)_i = \Psi(\bv)_j + v_{d_i}\cdot\Psi(\bv)_k.
	\end{equation}

	\item If $\tau_i(\C_0,\omega_0) = 2$, i.e. the line $\ell_i$ passes through two multiple points $P_0$ and $Q_0$ in $(\C_0)_{i}$. The line $\ell_i$ is then parametrized by
	\begin{equation}\label{eq:fixed_line}
		\Psi(\bv)_i = \Phi_{\gen{P_0}}(\bv) \times \Phi_{\gen{Q_0}}(\bv),
	\end{equation}
	where ``$\times$'' stands for the usual cross product.
\end{itemize}

\medskip

The next step is to construct the polynomials $\Delta_1,\ldots,\Delta_m$. Assume that the $i$th elementary perturbation $\C_{i-1}\prec\C_{i}$ in $\C_0\pprec\C(\A)$ is at $(\ell_j, P)$ and let $\widetilde{P}$ be the parent of $P$ in $\C_{i-1}$. Take two distinct elements $\ell_{i_1},\ell_{i_2}\in\widetilde{P}$ and define:
\begin{equation}\label{eq:Delta}
  \Delta_i(\bv) = \det\left(\Psi(\bv)_{i_1},\Psi(\bv)_{i_2},\Psi(\bv)_{j}\right)\in\CC[v_1,\ldots,v_{d_0}].
\end{equation}
It is worth noticing that another choice of $\ell_{i_1},\ell_{i_2}\in\widetilde{P}$ will lead to an equivalent construction of $\M(\A)$, due to Remark~\ref{remark:set_conditions}.

\medskip

The Zariski open subset $W_0$ can be expressed in a similar way as
\[
  W_0 = \set{\bv\in\CC^{d_0} \mid \det\left(\Psi(\bv)_{i},\Psi(\bv)_{j},\Psi(\bv)_{k}\right) \neq 0,\, \text{if $\not\exists P$ in $\C(\A)$ such that $\set{i,j,k}\subset P$}}.
\]
Since $\A\not\in\X(n)\cup\ol{\X}(n)$ by hypothesis, the above polynomial inequalities imply that $\ell_{i}\neq\ell_{j}$, for any $i\neq j \in\set{1,\ldots,n}$.

\subsection{Inductive upper bound}\mbox{}

The purpose of this section is to compute an upper bound for $\# \compo( \M(\A) )$, the number of connected components of $\M(\A)$. To achieve this goal, we construct an upper bound on the number of irreducible components of $\M(\A)$, and then we apply the following classical lemma, e.g.~\cite[Sec.~7.2]{Shafarevich2}.
\begin{lemma}\label{lem:cc_irr}
	Let $\A$ be a line arrangement. We have the following inequality:
	\[
		\# \compo( \M(\A) ) \leq \# \Irr( \M(\A) ),
	\]
	where $\Irr(\M(\A))$ is the set of irreducible components of $\M(\A)$.
\end{lemma}

When $\A$ is in $\X(n)$ or $\ol{\X}(n)$, we know by Remark~\ref{rmk:moduli_pencils} that its moduli space $\M(\A)$ is irreducible and so connected. Among this section, we assume that $\A\not\in \X(n)\cup\ol{\X}(n)$.\\

Let $\A=\{\ell_1,\dots,\ell_n\}$ be an arrangement with a $m$-perturbation $\C_0\pprec \C(\A)=(\A,\P)$ and perturbation order $\omega_0$. Up to relabelling, we assume that $\omega_0(\ell_i)=i$. Furthermore, we can also assume that $\ell_1$, $\ell_2$, $\ell_3$ and $\ell_4$ are in generic position in $\C(\A)$. We define recursively two applications $\Lambda:\A\rightarrow \NN$ and $\Theta:\P\rightarrow \NN$ as follows.

\begin{enumerate}[label=(R\arabic*), ref=(R\arabic*)]
	\item \label{AI_1} For $i\in\{1,2,3,4\}$, we fix $\Lambda(\ell_i)=0$.

	\item \label{AI_2} For any $P\in \C(\A_4)$, we fix $\Theta(\gen{P})=0$.

	\item \label{AI_3} For $i\in\{5,\dots,n\}$, the expression of $\Lambda(\ell_i)$ 
	depends on the value $\tau_i(\C_0,\omega_0)$, following the ideas of the proof in Section~\ref{sec:parametrization}:
	\begin{itemize}
		\item If $\tau_i(\C_0,\omega_0)=0$, we fix
		\[
			\Lambda(\ell_i) = 1.
		\]

		\item If $\tau_i(\C_0,\omega_0)=1$, i.e.~the line $\ell_i$ passes through a unique singular point $P_0$ in $(\C_0)_i$, we define:
		\[
			\Lambda(\ell_i) = \left\{
			\begin{array}{ll}
				1 + \Lambda(\ell_j) & \text{ if } \Lambda(\ell_j)=\Lambda(\ell_k),\\
				\max(\Lambda(\ell_j), \Lambda(\ell_k)) & \text{ otherwise. }
			\end{array}
			\right.
		\]
		where $\ell_j$ and $\ell_k$ are two lines in the parent of $P_0$.
		
		\item If $\tau_i(\C_0,\omega_0)=2$, i.e.~the line $\ell_i$ passes through two singular points $P_0$ and $Q_0$ in $(\C_0)_i$, we define
		\[
			\Lambda(\ell_i) = \Theta(\gen{P_0}) + \Theta(\gen{Q_0}).
		\]
	\end{itemize}

	\item \label{AI_4} For $i\in\{5,\dots,n\}$ and for any double point $P=\{\ell_i,\ell_j\}$ in $(\C_0)_i$, we define
	\[
		\Theta(\gen{P})=\Lambda(\ell_i)+\Lambda(\ell_j).
	\]
\end{enumerate}

\begin{lemma}\label{lem:deg_bound}
	Let $\A$ be an arrangement and let $\C_0\pprec\C(\A)$ be a $m$-perturbation. There exists a parametrization $\Psi$ of $\M(\A)$ as in \eqref{eq:Psi} such that for all $j\in\{1,\dots,n\}$:
	\begin{equation*}\label{eq:bound}
		\max(\deg a_j, \deg b_j, \deg c_j) \leq \Lambda(\ell_j).
	\end{equation*}
	Moreover, consider the description of $\M(\A)$ induced by $\Psi$ as in \eqref{eq:moduli_descritpion}. If $\C_{i-1} \prec \C_{i}$ is the elementary perturbation at $(\ell_j,P)$, and if $\ell_{j_1}$ and $\ell_{j_2}$ are two lines passing through the parent of $P$ in $\C_{i-1}$, then:
	\[
		\deg \Delta_i \leq \Lambda(\ell_{j_1}) + \Lambda(\ell_{j_2}) + \Lambda(\ell_j).
	\]
\end{lemma}

\begin{proof}
	The first inequality holds straightforward by construction of $\Lambda$ and $\Theta$. More precisely, the rules~\ref{AI_1},~\ref{AI_2} and~\ref{AI_4} control the behaviour of the degrees of the polynomials in equations~\eqref{eq:square_lines},~\eqref{eq:square_pts} and~\eqref{eq:intersection_points}, respectively, meanwhile~\ref{AI_3} controls the one of~\eqref{eq:generic_line}, \eqref{eq:pencil_line} and \eqref{eq:fixed_line}. Furthermore, by Equation~\eqref{eq:Delta}, $\Delta_i = \det(\Psi(\bv)_{j_1},\Psi(\bv)_{j_2},\Psi(\bv)_j)$, then the first inequality induces the second part of the lemma.
\end{proof}

\begin{theorem}\label{thm:upper_bound}
	Let $\A$ be a line arrangement and assume that $\C(\A)$ admits a $m$-perturbation $\C_0\pprec\C(\A)$. The following inequality holds:
	\[
		\# \compo ( \M(\A) ) \leq \prod_{i=1}^{m} \big( \Lambda(\ell_{i_1})+\Lambda(\ell_{i_2})+\Lambda(\ell_j) \big).
	\]
	where $\C_{i-1} \prec \C_{i}$ is the elementary perturbation at $(\ell_j,P)$, and $\ell_{j_1}$,~$\ell_{j_2}$ are two lines contained in the parent of $P$ in $\C_{i-1}$.
\end{theorem}

\begin{proof}
	In Theorem~\ref{thm:moduli_embedding}, we have obtain a description of $\M(\A)$ as the intersection of at most  $m$ proper algebraic hypersurfaces $\set{V_i:\Delta_i=0}_{i=1}^m$, in a non-empty Zariski open subset of an affine linear space. By~\cite[Ex. 8.3.6]{Fulton}, this implies that $\# \Irr ( \M(\A) ) \leq \prod_{i=1}^m \deg \Delta_i$. By Lemmas~\ref{lem:cc_irr} and~\ref{lem:deg_bound}, the result holds.
\end{proof}

\begin{remark}
	There are three types of choices in the construction of the upper bound in the previous theorem: the first one is obviously the choice of the $m$-perturbation $\C_0\pprec\C(\A)$, the second one is in rule~\ref{AI_3} (the choice of the lines $\ell_{i_1}$ and $\ell_{i_2}$ when $\tau_i(\C_0,\omega_0)=1$), and the third one is the choice of the lines $\ell_{j_1}$ and $\ell_{j_2}$ in the elementary perturbation $\C_{i-1} \prec \C_{i}$. Note that the value of the upper bound depends on the previous choices.\\
	Heuristically, we can optimize the upper bound by taking the lines $\ell_{i_1}, \ell_{i_2}$ and $\ell_{j_1}, \ell_{j_2}$ which minimize the value of their images by the map~$\Lambda$.
\end{remark}

\subsection{Sharpness of the upper bound}\label{sec:sharpness}\mbox{}

By construction, the inequality in Theorem~\ref{thm:upper_bound} becomes an equality for any inductively connected or inductively rigid arrangement. The question is then, is this inequality still sharp in non-trivial cases?

\begin{theorem}\label{thm:sharpness}
	For any $N\in\NN_{\geq2}$, there exists an arrangement $\A$ such that $\#\compo(\M(\A)) \geq N$, and
	\[
		\# \compo ( \M(\A) ) = \prod_{i=1}^{m} \big( \Lambda(\ell_{i_1})+\Lambda(\ell_{i_2})+\Lambda(\ell_j) \big),
	\]
	with the notation of Theorem~\ref{thm:upper_bound}.
\end{theorem}

\begin{proof}

	Let $p$ be a prime number such that $p\geq N$. Fix a primitive $p$-root of unity $\zeta$. Consider the arrangement $\A_p=\set{\ell_1,\ldots,\ell_{2p+2}}$ with lines:
	\[   
	  \ell_1\colon x-y=0,\quad \ell_2\colon x-\zeta y=0,\quad \ell_{2i+1}\colon x+\zeta^i z=0\quad\text{ and }\quad \ell_{2i+2}:\zeta^{-i}y+z=0,
	\]
	for $i\in\{1,\dots,p\}$. The multiple points of $\A_p$ are the two points of multiplicity $p$:
	\[   
	  \{\ell_3,\ell_5,\dots,\ell_{2p+1}\}\quad\text{and}\quad\{\ell_4, \ell_6, \dots, \ell_{2p+2}\},
	\]
	and the $2p$ triple points given by: 
	\[   
	  \{\ell_1,\ell_{2i+1},\ell_{2i+2}\}\quad\text{and}\quad\{\ell_2,\ell_{2i+2},\ell_{[2i]+3}\}
	\]
	for $i\in\{1,\dots,p\}$, where $[a]$ is the value of $a$ modulo $2p$ such that $0\leq [a]<2p$.

	Consider the elementary perturbation $\widetilde{\C}(\A_p)$ on $\C(\A_p)$ at $(\ell_{2p+2},\{\ell_2,\ell_3, \ell_{2p+2}\})$. Let $\omega_0$ be the following order on $\widetilde{\C}(\A_p)$:
	\[
		(\ell_1,\dots,\ell_{2p+2}) \longmapsto (5,6,1,2,3,4,7,8,\dots,(2p+2)).
	\]
	We have that $\tau(\widetilde{\C}(\A_p),\omega_0)=(0,0,0,0,2,1,2,2\dots,2,2)$. So $\widetilde{\C}(\A_p)$ is  inductively connected, and thus it is a $1$-perturbation of $\C(\A_p)$. Using the rules~\ref{AI_1}-\ref{AI_2}-\ref{AI_3}-\ref{AI_4}, we can compute that:
	\[
		\Lambda_p : (\ell_1,\dots,\ell_{2p+2}) \longmapsto (0,1,0,0,0,0,1,1,2,2,\dots,p-2,p-2).
	\]
	By Theorem~\ref{thm:upper_bound}, we obtain that $\M(\A_p)$ has at most $\Lambda_p(\ell_2)+\Lambda_p(\ell_3)+\Lambda_p(\ell_{2p+2})=1+0+(p-2)=p-1$ connected components. %
	On the other hand, the computation of the moduli space using Theorem~\ref{thm:moduli_embedding} shows that the dimension is zero. Furthermore, it contains $\A_p$ for any choice of primitive $p$-root of unity $\zeta$. We conclude that $\# \compo (\M(\A_p)) = p-1$.
\end{proof}

Note that the arrangement $\A_2$ corresponds to the Ceva arrangement, while the arrangements $\A_3$ are the MacLane arrangements~\cite{MacLane}. This family of arrangements appears in~\cite{Cadegan} in the context of Zariski pairs as subarrangements of the reflection arrangements associated with $G(N,N,3)$.

\subsection{Class of $\kappa_m$ arrangements}\mbox{}

In the spirit of the class of $C_k$ arrangements, let us introduce the following class.

\begin{definition}
	An arrangement $\A$ is $\kappa_m$ if $m$ is the smaller integer such that $\C(\A)$ admits a $m$-perturbation.
\end{definition}

\begin{example}
	The class of $\kappa_0$ arrangements corresponds to the class of inductively connected arrangements, and the arrangements $\A_p$ defined in Section~\ref{sec:sharpness} are all $\kappa_1$.
\end{example}

For fixed $n\in\NN$, we define $\mathfrak{C}_m(n)$ as the maximal number of connected components of the moduli space among all $\kappa_m$ arrangements of $n$ lines. 
It is clear that $\mathfrak{C}_0(n)=1$ for all $n\in\NN_{>0}$. %
Note that if $\A$ is a $\kappa_m$ arrangement, then the arrangement $\A\cup\{\ell\}$ is also $\kappa_m$ for any line $\ell$ which is generic with $\A$. It follows that $\mathfrak{C}_m(n+1)\geq \mathfrak{C}_m(n)$.\\

\noindent\textbf{Question.} For fixed $m$, how the number $\mathfrak{C}_m(n)$ behaves when $n$ grows?\\

As it is shown in the proof of Theorem~\ref{thm:sharpness}, we know that $\mathfrak{C}_1(n)$ grows at least as $n/2$.

\section*{Acknowledgments}

The authors would like to thank Prof. Nazir and Prof. Yoshinaga for inspiring this work through their article about inductively connected arrangements. In particular, we thank Prof. Yoshinaga for the useful remarks and helpful discussions.


\bibliographystyle{alpha}
\bibliography{biblio}

\end{document}